\documentclass[11pt,reqno]{amsart}
\usepackage{amsthm,amssymb,amsmath}
\usepackage{textcomp}
\usepackage{xcolor}
\usepackage{ulem}
\usepackage{esint}

\usepackage[T1]{fontenc}
\usepackage{lmodern}

\newtheorem{theorem}{Theorem}
\newtheorem*{theorem*}{Theorem}

\newtheorem{proposition}{Proposition}

\theoremstyle{definition}

\newtheorem*{definition*}{\bf Definition}

\newtheorem{example}{\sc Example}

\newtheorem{remark}{\sc Remark}
\newtheorem*{remark*}{\sc Remark}

\newtheorem*{example*}{\sc Example}

\newcommand{\loc}{{\rm loc}}

\begin{document}

\title[Strong solutions and form-bounded drifts]{Strong solutions of SDEs with singular (form-bounded) drift via R\"{o}ckner-Zhao approach}

\author{D.\,Kinzebulatov}

\author{K.R.\,Madou}

\email{damir.kinzebulatov@mat.ulaval.ca}

\email{kodjo-raphael.madou.1@ulaval.ca}

\address{Universit\'{e} Laval, D\'{e}partement de math\'{e}matiques et de statistique, Qu\'{e}bec, QC, Canada}



\thanks{The research of D.K. is supported by the NSERC (grant RGPIN-2017-05567).}

\begin{abstract}
We use the approach of R\"{o}ckner-Zhao to prove strong well-posedness for SDEs with singular drift satisfying some minimal assumptions.
\end{abstract}

\maketitle

\section{Introduction and result}

\textbf{1.~}Consider stochastic differential equation (SDE)
\begin{equation}
\label{sde1}
X_t^x=x+\int_0^t b(s,X_s^x)ds + W_t, \quad 0 \leq t \leq T,
\end{equation}
where $x \in \mathbb R^d$, $d \geq 3$, $b:\mathbb R^{d+1} \rightarrow \mathbb R^d$ is a Borel measurable vector field (drift), and $\{W_t\}_{0 \leq t \leq T}$ is a Brownian motion on a complete filtered probability space $(\Omega,\{\mathcal F_t\}_{0 \leq t \leq T},\mathcal F,\mathbf{P})$.

\medskip

One of the central problems in the theory of diffusion processes is the problem of strong well-posedness of SDE \eqref{sde1} under minimal assumptions on a locally unbounded drift $b$, for every starting point $x \in \mathbb R^d$. The following are the milestone results. Veretennikov \cite{V} was first who proved strong well-posedness of \eqref{sde1} for discontinuous drifts $b \in L^\infty(\mathbb R \times \mathbb R^d)$. Krylov-R\"{o}ckner \cite{KR} established strong well-posedness assuming that the drift in the sub-critical  
Ladyzhenskaya-Prodi-Serrin class 
\begin{equation}
\label{LPS0}
b \in L^p(\mathbb R,L^q(\mathbb R^d)), \quad \frac{d}{q}+\frac{2}{p} < 1, \quad p > 2, \quad q > d.
\end{equation}
Beck-Flandoli-Gubinelli-Maurelli \cite{BFGM} established strong existence and uniqueness for drifts in the critical Ladyzhenskaya-Prodi-Serrin class 
\begin{equation}
\label{LPS}
\tag{LPS}
b \in L^p(\mathbb R,L^q(\mathbb R^d)), \quad \frac{d}{q}+\frac{2}{p} \leq 1, \quad p \geq 2, \quad q \geq d,
\end{equation}
but only for a.e.\,starting point $x \in \mathbb R^d$. A major step forward was made recently by R\"{o}ckner-Zhao \cite{RZ} who established strong existence and uniqueness for  \eqref{sde1}
with drift $b$ in the critical Ladyzhenskaya-Prodi-Serrin class \eqref{LPS} ($p>2$) for every $x \in \mathbb R^d$. 
Another major advancement is the series of papers \cite{Kr1,Kr2,Kr3,Kr_strong} 
where Krylov proved strong well-posedness of \eqref{sde1}, for every $x \in \mathbb R^d$, 
for $|b| \in L^d$ and beyond, in a large Morrey class of time-inhomogeneous drifts 
(in terms of the Morrey norm \eqref{morrey_elliptic_2}, one has to have $\|b\|_{M_{s}}$, $s>\frac{d}{2} \vee 2$, sufficiently small).

The method of R\"{o}ckner-Zhao is different from the methods used in the other cited papers, and is based on a relative compactness criterion for random fields on the Wiener-Sobolev space. Their proof of uniqueness uses Cherny's theorem \cite{C} (strong existence + weak uniqueness $\Rightarrow$ strong uniqueness). The method of \cite{RZ} is a far-reaching strengthening of the methods of Meyer-Brandis and Proske \cite{MP}, Mohammed-Nilsen-Proske \cite{MNP} (for $b \in L^\infty(\mathbb R \times \mathbb R^d)$) and Rezakhanlou \cite{R} (for $b$ in \eqref{LPS0}). We refer again to \cite{RZ} for  a comprehensive survey of these and other important results on strong well-posedness of SDE \eqref{sde1}.

\medskip

We show in this paper that the method of R\"{o}ckner-Zhao  works, with few modifications, for a larger  class of form-bounded drifts. 
Together with the weak uniqueness result from \cite{KM}, their method yields strong well-posedness of SDE \eqref{sde1} with form-bounded drift (Theorem \ref{thm1}).

\begin{definition*}
A locally square integrable vector field $b:\mathbb R^{d+1} \rightarrow \mathbb R^d$  is said to be form-bounded  if 
there exist a constant $\delta>0$ such that for a.e.\,$t \in \mathbb R$ the following quadratic form inequality holds:
\begin{align}
\label{fbb_inhom}
\|b(t,\cdot)\varphi\|_2^2   \leq \delta \|\nabla \varphi\|_2^2 +g_\delta(t)\|\varphi\|_2^2
\end{align}
for all $\varphi \in W^{1,2}$, for some function $0 \leq g_\delta \in L^1_{\loc}(\mathbb R)$.
\end{definition*}

Throughout the paper, $\|\cdot\|_p$ denotes the norm in the Lebesgue space $L^p:=L^p(\mathbb R^d,dx)$; $W^{1,p}:=W^{1,p}(\mathbb R^d,dx)$ is the Sobolev space. 

\medskip

Condition \eqref{fbb_inhom} will be written as $b \in \mathbf{F}_\delta.$ This is essentially the largest class of vector fields $b$, defined in terms of $|b|$, that provides an $L^2$ theory of 
divergence-form operator $-\nabla \cdot a \cdot \nabla  + b \cdot \nabla$. See \cite{Ki_survey} for detailed discussion.

\begin{example}
The critical Ladyzhenskaya-Prodi-Serrin class  \eqref{LPS}
is contained in the class of form-bounded vector fields. For $q=d$ and $p=\infty$ this is an immediate consequence of the Sobolev embedding theorem:
$$
\|b(t,\cdot)\varphi\|_2^2  \leq \|b(t,\cdot)\|_d^2 \|\varphi\|_{\frac{2d}{d-2}}^2 \leq C_S \|b(t,\cdot)\|_d^2 \|\nabla \varphi\|_2^2,
$$
so $\delta=C_S \sup_{t \in \mathbb R}\|b(t,\cdot)\|_d^2$ and $g_\delta=0$ 
 (for $q>d$ and $p<\infty$ using, additionally, a simple interpolation argument, in which case $g$ is in general non-zero, see e.g.\,\cite{KM} for the proof). Moreover, if e.g.\,$b \in C_c(\mathbb R,L^d(\mathbb R^d))$, then form-bound $\delta$ can be chosen arbitrarily small at expense of increasing $g_\delta$.

\end{example}

\begin{example}
\label{morrey_ex}
Another subclass of \eqref{fbb_inhom}, which is considerably larger than $L^\infty(\mathbb R,L^d)$, consists of vector fields $b$ such that $b(t,\cdot)$ belongs, uniformly in $t \in \mathbb R$, to the scaling-invariant Morrey class $M_{2+\varepsilon}$. That is,
\begin{equation}
\label{morrey_elliptic_2}
\sup_{t \in \mathbb R}\|b(t,\cdot)\|_{M_{2+\varepsilon}}=\sup_{t \in \mathbb R}\sup_{r>0, x \in \mathbb R^d} r\biggl(\frac{1}{|B_r|}\int_{B_r(x)}|b(t,\cdot)|^{2+\varepsilon}dx \biggr)^{\frac{1}{2+\varepsilon}}<\infty
\end{equation}
where $B_r(x)$ is the ball of radius $r$ centered at $x$, and $\varepsilon$ is fixed arbitrarily small. 
Then, by a result in \cite{F} (see also \cite{CF}), 
$$
b \in \mathbf{F}_\delta  \quad \text{with $\delta=C\sup_{t \in \mathbb R}\|b(t,\cdot)\|_{M_{2+\varepsilon}}$} \text{ and } g_\delta=0
$$
for appropriate constant $C$. Note that Morrey $M_s$ becomes larger as $s$ becomes smaller.
\end{example}

\begin{example}
Morrey class \eqref{morrey_elliptic_2} contains vector fields $b$ with
$
\|b\|_{L^\infty(\mathbb R,L^{d,w})}<\infty.
$

\medskip

Recall that the norm in the weak $L^d$ space is defined as $$\|h\|_{L^{d,w}}:=\sup_{s>0}s|\{x \in \mathbb R^d: |h(x)|>s\}|^{1/d}.$$ 
(Clearly, $L^d \subset L^{d,w}$, but not vice versa, e.g.\,$h(x)=|x|^{-1}$ is in $L^{d,w}$ but not in $L^d$.)
\end{example}

Let us add that the attracting drift $$b(x)=-\frac{d-2}{2}\sqrt{\delta}|x|^{-2}x,$$ 
which is contained\footnote{and not contained in any $\mathbf{F}_{\delta'}$ with $\delta'<\delta$ regardless of the choice of $g_{\delta'}$} in $\mathbf{F}_\delta$ with $g_\delta=0$ (and is contained in Examples 2 and 3, but not in Example 1) has critical singularity at the origin. That is, if $\delta>0$ is too large, then SDE \eqref{sde1} with starting point $x=0$ does not even have a weak solution.  But, if $\delta$ is sufficiently small, then this SDE is strongly well-posed, see Theorem \ref{thm1}. (In fact, the critical value of $\delta$ for weak solvability, at least in high dimensions, is $\delta=4$, see \cite{KiS_sharp}.)

\medskip

An equivalent form of the a.e.\,inequality \eqref{fbb_inhom} is: for every $-\infty<t_1<t_2<\infty$,
$$
\int_{t_1}^{t_2} \|b(t)\psi(t)\|_2^2 dt \leq \delta \int_{t_1}^{t_2} \|\nabla \psi(t)\|_2^2 dt + \int_{t_1}^{t_2} g_\delta (t) \|\psi(t)\|_2^2 dt
$$
for all $\psi \in L^\infty(\mathbb R,W^{1,2})$.

\medskip

The class of form-bounded drifts is well known in the literature on parabolic equations, see Sem\"{e}nov \cite{S} and references therein.

\bigskip

\textbf{2.~}Our goal here is to prove a principal result: the SDE \eqref{sde1} with drift $b$ having form-bounded singularities is strongly well-posed. So, we will require in this paper, for simplicity,

\medskip

($A$)  $b$ has compact support and $g_\delta=0$ (the last assumption can be removed, see Remark \ref{g_rem}).

\medskip

Fix $T>0$.

\begin{theorem}
\label{thm1}
Let $d \geq 3$. Assume that $b \in \mathbf{F}_\delta$ and satisfies ($A$). Then, provided that form-bound $\delta$ is sufficiently small, for every $x \in \mathbb R^d$, SDE \eqref{sde1} has a strong solution $X_t^x$. This strong solution satisfies the following Krylov-type bounds:

{\rm 1)} For a given $q \in ]d,\delta^{-\frac{1}{2}}[$ and any vector field $\mathsf{g} \in \mathbf{F}_{\delta_1}$, $\delta_1<\infty$,
\begin{equation}
\label{krylov0}
\mathbf E \int_0^T |\mathsf{g}h|(\tau,X_{0,\tau}^x)d\tau \leq c\|\mathsf{g}|h|^{\frac{q}{2}}\|^{\frac{2}{q}}_{L^2([0,T] \times \mathbb R^d)} \quad \text{ for all }h \in C_c([0,T] \times \mathbb R^d).
\end{equation}

{\rm 2)} For a given $\mu>\frac{d+2}{2}$, there exists constant $C$ such that
\begin{equation}
\label{krylov}
\mathbf{E}\bigg[\int_{0}^{T}|h(\tau,X_{0,\tau}^{x})|d\tau\bigg] \leq C\|h\|_{L^{\mu}([0,T] \times \mathbb R^d)} \quad \text{ for all }h \in C_c([0,T] \times \mathbb R^d).
\end{equation}

Solution $X_t^x$ is unique among strong solutions to \eqref{sde1} that satisfy \eqref{krylov0} for some $q \in ]d,\delta^{-\frac{1}{2}}[$ with $\mathsf{g}=1$ and with $\mathsf{g}=b$. 

If, in addition to our hypothesis on $b$, one has $|b| \in L^{\frac{d+2}{2}+\varepsilon}$ for some $\varepsilon>0$, then $X_t^x$ is unique among strong solutions to \eqref{sde1} that satisfy \eqref{krylov}.
\end{theorem}

The proof of Theorem \ref{thm1} follows closely \cite{RZ}, except the proof of Proposition \ref{prop1} (this is Lemma 4.2(a) in \cite{RZ}). In \cite{RZ}, this result is proved using Sobolev regularity estimates for solutions of parabolic equations with distributional right-hand side (these estimates, developing earlier work of Krylov, are quite strong and are interesting on their own). We prove Proposition \ref{prop1} using a simpler argument which uses weaker estimates on solutions of parabolic equations, and thus allows to treat a larger class of form-bounded drifts. We also use some estimates from paper \cite{KM} that deals with weak well-posedness of SDE \eqref{sde1} with drift $b \in \mathbf{F}_\delta$.

\medskip

It should be added that for the drifts $b \in C([0,T],L^d)$ or $b \in \eqref{LPS}$ ($2<p<\infty$) considered in \cite{RZ} the form-bound $\delta$ can be chosen arbitrarily small. In other words, replacing drift $b$ by $cb$, for arbitrarily large constant $c$, does not affect strong well-posedness of SDE \eqref{sde1}. The latter is important in \cite{RZ} since they apply their strong well-posedness result to  Navier-Stokes equations.

\medskip

One can also prove strong well-posedness of SDE \eqref{sde1} with form-bounded drift $b=b(x)$ using the approach of \cite{BFGM}, but only for a.e.\,$x \in \mathbb R^d$, see \cite{KiSS_transport}.

\begin{remark}[On weak solutions] Weak existence and uniqueness for \eqref{sde1} is known to hold for larger classes of drifts than the class $\mathbf{F}_\delta$, see \cite{KiS_brownian} dealing with weakly form-bounded drifts (time-homogeneous case) and \cite{Ki_Morrey} dealing with time-inhomogeneous drifts in essentially the largest possible Morrey class.  See also \cite{RZ_weak}. In a recent paper \cite{Kr_Morrey_sdes}, Krylov proved weak existence and uniqueness for SDEs with VMO diffusion coefficients and time-inhomogeneous drift in a large Morrey class containing \eqref{LPS} (in terms of Example \ref{morrey_ex}, this is the Morrey class with exponent $2+\varepsilon$ replaced by $\frac{d}{2}+\varepsilon$; note that in dimension $d=3$ Krylov's Morrey class is larger than $\mathbf{F}_\delta$). 
We refer to \cite{RZ_weak} for a survey of the literature on weak solutions of \eqref{sde1}.
\end{remark}

\bigskip

\section{Proof of Theorem \ref{thm1}}

\subsection{Notations} Set $\Delta_n(T_0,T_1):=\{(t_1,\dots,t_n) \mid T_0 \leq t_1 \leq \dots \leq t_n\ \leq T_1\}$ and put $\Delta_n(T):=\Delta_n(0,T)$.

Let $\nabla_{i}:=\partial_{x_i}$, $x=(x_1,\dots,x_d) \in \mathbb R^d$. 

Let $\mathbf{E}_{\mathcal F_t}$ denote conditional expectation with respect to $\sigma$-algebra $\mathcal F_t$.

Put
$$
\langle f,g\rangle = \langle f g\rangle :=\int_{\mathbb R^d}f gdx.$$

\subsection{Some estimates}
Let $f_i \in L^2_{\loc}(\mathbb R^{d+1})$ ($i \geq 1$) be form-bounded:
\begin{equation}
\label{fbb_fi}
\|f_i(t,\cdot)\varphi\|_2^2   \leq \nu \|\nabla \varphi\|_2^2
\end{equation}
for some $\nu>0$. Also, in this section, $f_i$ are smooth. Additionally, let us assume that:

\medskip

($A'$) all $f_i$  have compact supports contained in $\mathbb R \times B_R(0)$ for a fixed $R>0$ (independent of $i$). 

\medskip

In this subsection, $b \in \mathbf{F}_\delta$ is additionally assumed to be smooth. However, the constants in the estimates below will not depend on smoothness or boundedness of $b$ and $f_i$.

\medskip

By the classical theory, there exists a unique strong solution $X_t^x$ to
$$
X_t^x=x+\int_0^t b(\tau,X_\tau^x)d\tau + W_t.
$$

\medskip

\medskip

Let $0 \leq T_0 \leq T_1 \leq T$.

\begin{proposition}
\label{prop1}
There exist positive constants $C_0$, $K$ such that, for every $n \geq 1$,
$$
\int_{\mathbb R^d }\left|\mathbf{E} \int_{\Delta_{n}(T_0,T_1)} \prod_{i=1}^n \nabla_{\alpha_i} f_i(t_i, X_{t_i}^x)dt_1 \dots dt_n \right|^2 dx \leq C_0K^n(T_1-T_0),
$$
where $1 \leq \alpha_i \leq d$ ($i \geq 1$).
Moreover, $K$ can be made as small as needed by assuming that form-bounds $\delta$ and $\nu$ in \eqref{fbb_inhom}, \eqref{fbb_fi} are sufficiently small.
\end{proposition}

\begin{proof}
Fix $n$, put $u_{n+1}=1$ and define consecutively $$g_k=(\nabla_{\alpha_k} f_k)u_{k+1}, \quad k=1,\dots,n,$$
where $u_k$ solves the terminal-value problem on $[T_0,T_1]$
\begin{equation}
\label{eq_uk}
\partial_t u_k + \frac{1}{2}\Delta u_k + b \cdot \nabla u_k+g_k=0, \quad 
u_k(T_1)=0.
\end{equation}
Then, repeating the argument in \cite[Proof of Lemma 4.2]{RZ},
$$
\mathbf{E}_{\mathcal F_{T_0}} \int_{\Delta_{n}(T_0,T_1)} \prod_{i=1}^n \nabla_{\alpha_i} f_i(t_i, X_{t_i}^x)dt_1 \dots dt_n=u_1(T_0,X_{T_0}^x).
$$
Again as in \cite{RZ}, let $U$ be the solution to the initial-value problem on $[0,T_1]$,
\begin{equation}
\label{eq1}
\partial_t U - \frac{1}{2}\Delta U - B \cdot \nabla U - G=0, \quad U(0)=0,
\end{equation}
where
$$
B(t,\cdot)=b(T_1-t,\cdot)\mathbf{1}_{[0,T_1-T_0]}(t) + b(t+T_0-T_1,\cdot)\mathbf{1}_{]T_1-T_0,T_1]}(t), 
$$
$$
 G(t,\cdot):=g_1(T_1-t,\cdot)\mathbf{1}_{[0,T_1-T_0]}(t).
$$
One has
$
U(t,\cdot)=u_1(T_1-t,\cdot)$, $t \in [0,T_1-T_0]
$.
Further, $V(t,x):=U(t+T_1-T_0)$ solves on $[0,T_0]$
$$
\partial_t V -\frac{1}{2}\Delta V - b \cdot \nabla V=0, \quad V(0,\cdot)=U(T_1-T_0,\cdot)=u_1(T_0,\cdot).
$$
Therefore, 
\begin{align*}
&\int_{\mathbb R^d }\left|\mathbf{E} \int_{\Delta_{n}(T_0,T_1)} \prod_{i=1}^n \nabla_{\alpha_i} f_i(t_i, X_{t_i}^x)dt_1 \dots dt_n \right|^2 dx \\
&=\int_{\mathbb R^d } |\mathbf{E}u_1(T_0,X_{T_0}^x)|^2 dx = \int_{\mathbb R^d } |V(T_0,x)|^2 dx = \|U(T_1,\cdot)\|_2^2.
\end{align*}
We estimate $\|U(T_1,\cdot)\|_2^2$ in three steps:

1.~We multiply equation \eqref{eq1} by $U$ and integrate over $[0,T_1] \times \mathbb R^d$, arriving at 
\begin{align}
\label{intEq}
\frac{1}{2}\langle U^2(T_1,\cdot)\rangle - 0  + \frac{1}{2}\int_0^{T_1}\langle |\nabla U|^2 \rangle ds &=  \int_0^{T_1} \langle B \cdot \nabla U,U  \rangle ds \\ \nonumber
&+ \int_0^{T_1} \langle g_1(T_1-s,\cdot)\mathbf{1}_{[0,T_1-T_0]}(s),U(s)\rangle ds.
\end{align}
The first term in the RHS  of \eqref{intEq} is estimated, using  the quadratic inequality $ac \leq \frac{1}{2\sqrt{\delta}}a^2 + \frac{\sqrt{\delta}}{2}c^2$ and the form-boundedness $b  \in \mathbf{F}_\delta$, as follows:
\begin{align}
 \int_0^{T_1} \langle B \cdot \nabla U,U  \rangle ds & \leq \frac{1}{2\sqrt{\delta}}  \int_0^{T_1} \langle B^2,U^2 \rangle ds + \frac{\sqrt{\delta}}{2} \int_0^{T_1}\langle |\nabla U|^2 \rangle ds \notag \\ 
& \leq \sqrt{\delta} \int_0^{T_1}\langle |\nabla U|^2 \rangle ds. \label{B_est}
\end{align} 
The second term in the RHS of \eqref{intEq}:
\begin{align*}
 \int_0^{T_1} \langle g_1(T_1-s,\cdot)& \mathbf{1}_{[0,T_1-T_0]}(s),U(s)\rangle  ds  = \int_0^{T_1} \langle \nabla_{\alpha_1} f_1(T_1-s, \cdot), u_2(T_1-s,\cdot)\mathbf{1}_{[0,T_1-T_0]}(s) U(s)\rangle ds \\
& = - \int_0^{T_1} \langle f_1(T_1-s, \cdot), (\nabla_{\alpha_1}  u_2(T_1-s,\cdot))\mathbf{1}_{[0,T_1-T_0]}(s) U(s,\cdot)\rangle ds \\
& -  \int_0^{T_1} \langle f_1(T_1-s, \cdot),  u_2(T_1-s,\cdot)\mathbf{1}_{[0,T_1-T_0]}(s) \nabla_{\alpha_1}  U(s,\cdot)\rangle ds \\
& (\text{we are applying quadratic inequality twice; fix some $\varepsilon$, $\beta>0$}) \\
& \leq  \varepsilon \int_0^{T_1} \langle f_1^2(T_1-s, \cdot) U^2(s,\cdot) \rangle ds + \frac{1}{4\varepsilon} \int_0^{T_1} \langle |\nabla_{\alpha_1}  u_2(T_1-s,\cdot)|^2\mathbf{1}_{[0,T_1-T_0]}(s) \rangle ds \\
&+ \beta\int_0^{T_1} \langle f_1^2(T_1-s, \cdot), u^2_2(T_1-s,\cdot)\mathbf{1}_{[0,T_1-T_0]}(s)  \rangle ds + \frac{1}{4\beta} \int_0^{T_1} \langle |\nabla_{\alpha_1}  U(s,\cdot)|^2\rangle ds.
\end{align*}
Therefore, taking into account the indicator function of $[0,T_1-T_0]$, and using the form-boundedness assumption \eqref{fbb_fi} on $f_i$, we obtain
\begin{align}
  \int_0^{T_1} \langle g_1(T_1-s,\cdot) \mathbf{1}_{[0,T_1-T_0]}(s), &U(s,\cdot)\rangle  ds  
 \leq \varepsilon \int_0^{T_1} \langle f_1^2(T_1-s, \cdot) U^2(s) \rangle ds + \frac{1}{4\varepsilon} \int_{T_0}^{T_1} \langle |\nabla_{\alpha_1}  u_2(s,\cdot)|^2 \rangle ds  \notag \\
&+ \beta\int_{T_0}^{T_1} \langle f_1^2(s, \cdot), u^2_2(s,\cdot)\rangle ds + \frac{1}{4\beta} \int_0^{T_1} \langle |\nabla_{\alpha_1}  U(s,\cdot)|^2\rangle ds \label{U_fb}\\
&\leq \bigg(\varepsilon \nu + \frac{1}{4\beta} \bigg) \int_0^{T_1} \langle |\nabla  U(s,\cdot)|^2\rangle ds + \bigg(\beta \nu + \frac{1}{4\varepsilon} \bigg)\int_{T_0}^{T_1} \langle |\nabla  u_2(s,\cdot)|^2 \rangle ds. \notag
\end{align}
Thus, we obtain from \eqref{intEq}:
\begin{align*}
\frac{1}{2}\langle U^2(T_1)\rangle  &  + \biggl(\frac{1}{2}-\sqrt{\delta}-\varepsilon \nu-\frac{1}{4\beta} \biggr)  \int_0^{T_1} \langle |\nabla U(s)|^2\rangle ds 
\leq \biggl(\beta \nu + \frac{1}{4\varepsilon} \biggr) \int_{T_0}^{T_1} \langle |\nabla  u_2(s)|^2 \rangle ds.
\end{align*}
Now, selecting $\varepsilon$ and $\beta$ large, and requiring the form-bounds $\delta$ and $\nu$ to be sufficiently small, we arrive at
\begin{align}
\label{U_est}
\langle U^2(T_1)\rangle  + C_1\int_0^{T_1}\langle |\nabla U(s)|^2 \rangle ds\leq C_2 \int_{T_0}^{T_1} \langle |\nabla u_2(s)|^2\rangle  ds
\end{align}
for constants $0<C_2<C_1$ independent of smoothness or boundedness of $b$ and $f_i$. Moreover, it is clear that we can make $\frac{C_2}{C_1}$ arbitrarily small by selecting $\delta$ and $\nu$ even smaller.

\medskip

2.~Now, we repeat this procedure for $u_2$ in place of $U$. That is, we multiply equation \eqref{eq_uk} (for $k=2$) by $u_2$ and integrate over $[T_0,T_1] \times \mathbb R^d$ to obtain
$$
\frac{1}{2}\langle u_2^2(T_0)\rangle + \frac{1}{2}\int_{T_0}^{T_1}\langle |\nabla u_2|^2 \rangle ds =  \int_{T_0}^{T_1} \langle b \cdot \nabla u_2,u_2\rangle ds + \int_{T_0}^{T_1} \langle g_2,u_2\rangle ds.
$$
We estimate the first term in the RHS as in \eqref{B_est}, using quadratic inequality and the assumption $b \in \mathbf{F}_\delta$.
The second term in the RHS:
\begin{align*}
\int_{T_0}^{T_1} \langle g_2,u_2\rangle ds & =\int_{T_0}^{T_1} \langle (\nabla_{\alpha_2} f_2)u_3,u_2 \rangle ds \\
& \leq - \int_{T_0}^{T_1} \langle f_2,(\nabla_{\alpha_2}u_3)u_2\rangle ds - \int_{T_0}^{T_1} \langle f_2,u_3 \nabla_{\alpha_2}u_2\rangle ds \\
& \leq \varepsilon \int_{T_0}^{T_1} \langle f_2^2, u_2^2 \rangle ds + \frac{1}{4\varepsilon} \int_{T_0}^{T_1} \langle |\nabla_{\alpha_2}  u_3(s,\cdot)|^2 \rangle ds \\
&+ \beta\int_{T_0}^{T_1} \langle f_2^2, u^2_3\rangle ds + \frac{1}{4\beta} \int_{T_0}^{T_1} \langle |\nabla_{\alpha_2}  u_2|^2\rangle ds\\
& (\text{we are using $f_2 \in \mathbf{F}_\nu$}) \\
& \leq \bigg(\varepsilon \nu + \frac{1}{4\beta} \bigg) \int_{T_0}^{T_1} \langle |\nabla  u_3(s)|^2\rangle ds + \bigg(\beta \nu + \frac{1}{4\varepsilon} \bigg)\int_{T_0}^{T_1} \langle |\nabla  u_2(s)|^2 \rangle ds,
\end{align*}
as in the previous step. Thus, we arrive at
$$
\int_{T_0}^{T_1}\langle |\nabla u_2|^2 \rangle ds \leq \frac{C_2}{C_1} \int_{T_0}^{T_1} \langle |\nabla u_3|^2\rangle  ds.
$$
If $n > 3$, we repeat this $n-3$ more times:
$$
\int_{T_0}^{T_1}\langle |\nabla u_{2}|^2 \rangle  ds\leq \biggl(\frac{C_2}{C_1}\biggr)^{n-2} \int_{T_0}^{T_1} \langle |\nabla u_n|^2\rangle  ds
$$
and so, in view of \eqref{U_est},
$$
\langle U^2(T_1)\rangle \leq C_2  \biggl(\frac{C_2}{C_1}\biggr)^{n-2} \int_{T_0}^{T_1} \langle |\nabla u_n|^2\rangle  ds.
$$

\medskip

3.~Finally, we estimate $\int_{T_0}^{T_1}\langle |\nabla u_n(s)|^2\rangle  ds$. Arguing as above,
we have (recall that $u_{n+1}=1$)
\begin{align*}
\int_{T_0}^{T_1}\langle |\nabla u_n(s)|^2 \rangle ds & \leq C_3\int_{T_0}^{T_1} \langle \nabla_{\alpha_n} 
f_n(s,\cdot),  u_n(s,\cdot)\rangle ds \\
& = - C_3\int_{T_0}^{T_1} \langle  f_n(s,\cdot), \nabla_{\alpha_n} u_n(s,\cdot)\rangle ds \\
& (\text{we are applying quadratic inequality}) \\
& \leq C_4\int_{T_0}^{T_1} \langle f_n^2 \rangle ds + \frac{1}{2}\int_{T_0}^{T_1} \langle |\nabla u_n(s)|^2 \rangle ds \\
& (\text{we are using assumption ($A'$) that all $f_i$ have support in $B_R(0)$,} \\
& \text{and apply \eqref{fbb_fi} to $\int_{T_0}^{T_1} \langle f_n^2 \varphi^2\rangle ds \geq \int_{T_0}^{T_1} \langle f_n^2 \rangle ds$ for a smooth $\varphi \geq \mathbf{1}_{B_R(0)}$}) \\
& \leq C_5(T_1-T_0) + \frac{1}{2}\int_{T_0}^{T_1} \langle |\nabla u_n(s)|^2  \rangle ds.
\end{align*}
Thus, $\frac{1}{2}\int_{T_0}^{T_1}\langle |\nabla u_n(s)|^2 \rangle ds \leq  C_5(T_1-T_0)$.
Combining this with the previous estimate, we obtain
$
\langle U^2(T_1)\rangle \leq C_2 \bigl(\frac{C_2}{C_1} \bigr)^{n-2} 2 C_5 (T_1-T_0),
$
which gives the required estimate with $K:=\frac{C_2}{C_1}$.
\end{proof}

\begin{remark}
\label{g_rem}
Let us comment on what happens if in Theorem \ref{thm1} we assume that $g_\delta$ is non-zero. We have to assume that $$0 \leq g_\delta \in L^{1+\varepsilon}_{\loc}(\mathbb R), \quad \text{ for a fixed $\varepsilon>0$}.$$ (It should be added that this $\varepsilon>0$ does not allow to include completely the critical Ladyzhenskaya-Prodi-Serrin class \eqref{LPS} even with $p>2$ there, as is assumed in \cite{RZ}. It does include, however, the case that interests us the most: $p=\infty$, $q=d$. It also includes with case $p>2$, $q=\infty$). 

Only the proof of Proposition \ref{prop1} has to be changed, where we assume in \eqref{fbb_fi} $0 \leq g_\nu \in L^{1+\varepsilon}_{\loc}(\mathbb R)$. Then the estimate of Proposition \ref{prop1} changes to
\begin{equation}
\label{prod_est2}
\int_{\mathbb R^d }\left|\mathbf{E} \int_{\Delta_{n}(T_0,T_1)} \prod_{i=1}^n \nabla_{\alpha_i} f_i(t_i, X_{t_i}^x)dt_1 \dots dt_n \right|^2 dx \leq C'_0K^n(T_1-T_0)^{\frac{\varepsilon}{1+\varepsilon}},
\end{equation}
which does not affect the validity of the result of the proof. The proof of \eqref{prod_est2} goes as follows.
Put $F(t):=\lambda \int_0^t \big[g_\delta(s)+g_\nu(s)\big]ds$, where $\lambda$ is to be fixed sufficiently large (depending on the values of $\delta$ and $\nu$). We multiply equation \eqref{eq1} for $U$ by $e^{-F}$, obtaining
$$
\partial_t (e^{-F}U) + F'e^{-F}U - \frac{1}{2}\Delta e^{-F}U - B \cdot \nabla e^{-F}U - e^{-F}G=0, \quad U(0)=0,
$$
where $e^{-F}G=(\partial_{\alpha_1}f_1(T_1-t,\cdot))\mathbf{1}_{[0,T_1-T_0]}(t)e^{-F(t)}u_2(T_1-t,\cdot)$. 
After multiplying the previous equation by $U$, integrating and fixing $\lambda>\frac{1}{2\sqrt{\delta}} + \varepsilon + \beta$, one sees that the term
$$
\int_0^{T_1} \langle F'e^{-F}U^2 \rangle ds= \lambda \int_0^{T_1} \langle (g_\delta+g_\nu)e^{-F}U^2 \rangle ds
$$
will
absorb the ``new'' terms $\frac{1}{2\sqrt{\delta}}\int_0^{T_1}\langle g_\delta e^{-F} U^2\rangle ds$ and $(\varepsilon + \beta)\int_{0}^{T_1}\langle g_\nu e^{-F}U^2 \rangle ds$ that will now appear in \eqref{B_est} and \eqref{U_fb}. This will give us, instead of \eqref{U_est}, the estimate:
$$
\langle e^{-F(T_1)}U^2(T_1)\rangle  + C_1\int_0^{T_1}\langle e^{-F}|\nabla U|^2 \rangle ds\leq C_2 \int_{T_0}^{T_1} \langle e^{-\tilde{F}}|\nabla u_2(s)|^2\rangle  ds,
$$
where $\tilde{F}(t):=F(T_1-s)$.

In turn, the multiple $e^{-\tilde{F}(t)}$ factors through all equations \eqref{eq_uk} with the same effect of absorbing the ``new'' terms containing $g_\delta$ and $g_\nu$, that is, we get
$$
\int_{T_0}^{T_1}\langle e^{-\tilde{F}}|\nabla u_2|^2 \rangle ds \leq \frac{C_2}{C_1} \int_{T_0}^{T_1} \langle e^{-\tilde{F}}|\nabla u_3|^2\rangle  ds,
$$
and 
so on: 
$$
\int_{T_0}^{T_1}\langle |e^{-\tilde{F}}\nabla u_{2}|^2 \rangle  ds\leq \biggl(\frac{C_2}{C_1}\biggr)^{n-2} \int_{T_0}^{T_1} \langle e^{-\tilde{F}}|\nabla u_n|^2\rangle  ds.
$$
 Finally, $e^{-\tilde{F}}$ does not affect the estimate on $u_n$, only the constant $C_5$. Thus, we arrive at \eqref{prod_est2} with the same constant $K$ that does not depend on $g_\delta$ or $g_\nu$.
\end{remark}

\bigskip

For a given vector field $Y=(Y_i)_{i=1}^d:\mathbb R^{k} \rightarrow \mathbb R^m$, denote 
\begin{equation}
\label{nabla_Y}
\nabla Y=\nabla_x Y(x):=\left( 
\begin{array}{cccc}
\nabla_1 Y_1 & \nabla_2 Y_1 & \dots & \nabla_k Y_1 \\

& &  \dots & \\

\nabla_1 Y_m & \nabla_2 Y_m & \dots & \nabla_k Y_m
\end{array}
\right).
\end{equation}

\begin{proposition}
\label{prop2}
For every $r \geq 1$, there exist constants $K_1$, $K_2$ (independent of smoothness or boundedness of $b$) such that

\medskip

{\rm (\textit{i})} $\|\nabla X_t^x - I\|_{L^{2r}(\mathbb R^d,L^r(\Omega))} \leq K_1 t^{\frac{1}{2r}}$ for all $0 \leq t \leq T$;

\medskip

{\rm (\textit{ii})} $\|D_sX_t^x - I\|_{L^{2r}(\mathbb R^d,L^r(\Omega))} \leq K_1(t-s)^{\frac{1}{4r}}$ for a.e.\,$s \in [0,T]$ and $0 \leq s \leq t \leq T$;

\medskip

{\rm (\textit{iii})} $\|D_s X_t^x - D_{s'}X_t^x\|_{L^{2r}(\mathbb R^d,L^r(\Omega))} \leq K_2|s-s'|^{\frac{1}{4r}}$ for a.e.\,$s,s' \in [0,T]$ and $0 \leq s,s' \leq t \leq T$.
\end{proposition}

\begin{proof} The proof repeats \cite[Proof of Prop.\,4.1]{RZ} essentially word in word.
We give an outline of the proof of (\textit{i}).  Since $b$ is bounded and smooth, one has
$$
\nabla X_t^x-I=\int_0^t \nabla b(s,X_s^x)\nabla X_s^x ds.
$$
The goal is to iterate this identity, obtaining an expression for the left-hand side that one can control:
$$
\nabla X_t^x-I=\sum_{n=1}^\infty \int_{\Delta_n(t)} \prod_{i=1}^n \nabla b(t_i,X_{t_i}^x)dt_1\dots dt_n,
$$
so
\begin{equation}
\label{sum}
\|\nabla X_t^x-I\|_{L^{2r}(\mathbb R^d,L^r(\Omega))} \leq \sum_{n=1}^\infty \left\|\int_{\Delta_n(t)} \prod_{i=1}^n \nabla b(t_i,X_{t_i}^x) dt_1\dots dt_n \right\|_{L^{2r}(\mathbb R^d,L^r(\Omega))}.
\end{equation}

Let us estimate 
\begin{align*}
\|\int_{\Delta_n(t)} \prod_{i=1}^n \nabla b(t_i,X_{t_i}^x) dt_1\dots dt_n\|_{L^{2r}(\mathbb R^d,L^r(\Omega))}= \bigg[\int_{\mathbb R^d} \biggl[ \mathbf{E} \biggl( \int_{\Delta_n(t)} \prod_{i=1}^n \nabla b(t_i,X_{t_i}^x)dt_1\dots dt_n \biggr)^r\biggr]^2 dx\bigg]^{\frac{1}{2r}}.
\end{align*}
First, note that by subdividing $\Delta_n(t) \times \dots \times \Delta_n(t)$ ($r$ times) into sub-simplexes, and recalling definition \eqref{nabla_Y}, one can represent 
\begin{equation}
\label{int_1}
(\int_{\Delta_n(t)} \prod_{i=1}^n \nabla b(t_i,X_{t_i}^x) dt_1\dots dt_n)^r 
\end{equation}
as a sum of at most $rn$ terms of the form
\begin{equation}
\label{int_2}
\int_{\Delta_{rn}(t)} \prod_{i=1}^n \nabla_{\gamma_1} b_{\beta_1}(t_1,X_{t_1}^x)  \dots \nabla_{\gamma_{rn}} b_{\beta_{rn}}(t_{rn},X_{t_{rn}}^x) dt_1 \dots dt_{rn},
\end{equation}
so 
\begin{align*}
& \|\int_{\Delta_n(t)} \prod_{i=1}^n \nabla b(t_i,X_{t_i}^x) dt_1\dots dt_n\|_{L^{2r}(\mathbb R^d,L^r(\Omega))} \\& \leq \sum\bigg[\int_{\mathbb R^d} \biggl[ \sum_{\beta,\gamma}\mathbf{E} \int_{\Delta_{rn}(t)} \prod_{i=1}^n \nabla_{\gamma_1} b_{\beta_1}(t_1,X_{t_1}^x)  \dots \nabla_{\gamma_{rn}} b_{\beta_{rn}}(t_{rn},X_{t_{rn}}^x) dt_1 \dots dt_{rn} \biggr]^2 dx\bigg]^{\frac{1}{2r}} \\
& \leq \sum\bigg[\sum_{\beta,\gamma}\bigg\| \mathbf{E} \int_{\Delta_{rn}(t)} \prod_{i=1}^n \nabla_{\gamma_1} b_{\beta_1}(t_1,X_{t_1}^x)  \dots \nabla_{\gamma_{rn}} b_{\beta_{rn}}(t_{rn},X_{t_{rn}}^x) dt_1 \dots dt_{rn} \bigg\|_{L^2(\mathbb R^d)}\bigg]^{\frac{1}{r}},
\end{align*}
where both sums are finite (the first sum comes from the coordinate representation of the product of $n$ matrices $\nabla b(t_i,X^x_{t_i})$, the second sum contains $r^{rn}$ terms). Finally, applying Proposition \ref{prop1}, one obtains
\begin{align*}
\|\int_{\Delta_n(t)} \prod_{i=1}^n \nabla b(t_i,X_{t_i}^x) dt_1\dots dt_n\|_{L^{2r}(\mathbb R^d,L^r(\Omega))}
\leq (C_6 r)^n C_7^{\frac{1}{r}} K^{n} t^{\frac{1}{r}}.
\end{align*}
Now, recalling that $K$ can be made as small as needed by assuming that $\delta$ is sufficiently small, one has $\|\int_{\Delta_n(t)} \prod_{i=1}^n \nabla b(t_i,X_{t_i}^x) dt_1\dots dt_n\|_{L^{2r}(\mathbb R^d,L^r(\Omega))}
\leq C_8^n t^{\frac{1}{r}}$ for a positive constant $C_8<1$

Returning to \eqref{sum}, one obtains
$$
\|\nabla X_t^x-I\|_{L^{2r}(\mathbb R^d,L^r(\Omega))} \leq \sum_{n=1}^\infty C_8^n t^{\frac{1}{r}},
$$
as needed.

The Malliavin derivative $D_sX_t^x$ satisfies (see e.g.\,\cite{B})
$$
D_s X_t^x-I=\int_s^t \nabla b(\tau,X_\tau^x)D_s X_\tau^x d\tau,
$$
so one can iterate this identity and estimate $D_s X_t^x-I$ in the same way as above, which yields (\textit{ii}). The latter yields (\textit{iii}), see \cite[Proof of Prop.\,4.1]{RZ} for details.
\end{proof}

\subsection{Proof of Theorem \ref{thm1}} The proof repeats the argument in \cite{RZ}. However, since we will have to use some estimates and some convergence results established in \cite{KM}, we included the details for the ease of the reader. 

\medskip

We consider a general $b \in \mathbf{F}_\delta$ as in the assumptions of the theorem.
Let us fix an approximation $\{b_m\} \subset C^\infty(\mathbb R^{d+1},\mathbb R^d) \cap L^\infty(\mathbb R^{d+1},\mathbb R^d)$ of $b$:
\begin{equation}
\label{b_m_cond1}
b_m \rightarrow b \quad \text{ in $L^2_{\loc}(\mathbb R^{d+1}, \mathbb R^d)$ as $m \rightarrow \infty$} 
\end{equation}
and for all $t \in \mathbb R$
\begin{equation}
\label{b_m_cond2}
\|b_m(t)\varphi\|_2^2   \leq \delta \|\nabla \varphi\|_2^2
\end{equation}

\begin{example*}
It is easy to show that the following $b_m$, with $\varepsilon_m \downarrow 0$ sufficiently rapidly and $c_m \uparrow 1$ sufficiently slow, satisfy \eqref{b_m_cond1}, \eqref{b_m_cond2}. 
$$
b_m:=c_m E^{1+d}_{\varepsilon_m} (\mathbf 1_m b),
$$
where $\mathbf 1_m$ is the indicator of $\{(t,x) \in [0,T] \times \mathbb R^d \mid |b(t,x)| \leq m\}$,  $E^{1+d}_\varepsilon$ is the Friedrichs mollifier on $\mathbb R \times \mathbb R^d$, see details e.g.\,in \cite{KM}. Note that, by selecting $\varepsilon_n \downarrow 0$ rapidly, one can treat $b_m$ as essentially a cutoff of $b$.
\end{example*}

 (Of course, since by our assumption $b$ in this paper has compact support, in \eqref{b_m_cond1} one has convergence in $L^2(\mathbb R^{d+1},\mathbb R^d)$.)

\medskip

Let $\mathsf{f} \in \mathbf{F}_\nu$ be bounded and smooth with function $g_\nu=0$.
(Below we will need $\mathsf{f}=b_m$, in which case $\nu=\delta$, or $\mathsf{f}=b_m-b_k$, in which case $\nu=2\delta$.) Let us emphasize  that the constants in the estimates below do not depend on $n$ or boundedness or smoothness of $\mathsf{f}$. They will depend on the dimension $d$, $T$ and form-bounds $\delta$ and $\nu$.

\medskip

Since $b_n$ are bounded and smooth, by the classical theory there exists a unique continuous random field $X^m:\Delta_2(T) \times \mathbb R^d \times \Omega \rightarrow \mathbb R^d$ such that
\begin{equation}
\label{sde2}
X_{s,t}^{x,m}=x+\int_s^t b_m(s,X^{x,m}_{s,r})dr + W_t-W_s, \quad 0 \leq s \leq t \leq T, \quad x \in \mathbb R^d.
\end{equation}

By It\^{o}'s formula, for all $m,k=1,2,\dots$, $s \leq t_1 \leq t_2 \leq T$,
$$
-u_m(t_1,X_{s,t_1}^{x,m})=-\int_{t_1}^{t_2}|\mathsf{f}(t,X_{s,t}^{x,m})|dt+\int_{t_1}^{t_2}\nabla u_m(t,X_{s,t}^{x,m})dW_t,
$$
where $u_m(t)$, $s<t \leq t_2$ is the classical solution to
$$
\partial_t u_m+\frac{1}{2}\Delta u_m + b_m\cdot \nabla u_m=-|\mathsf{f}|, \quad u_m(t_2)=0.
$$

\medskip

The following estimates on $X^{x,m}_{s,t}$ and $u_m$ are valid:
\medskip

1) $$\sup_{m}\sup_{x \in \mathbb R^d}\mathbf{E}\big[\int_{t_1}^{t_2}|\mathsf{f}(t,X_{s,t}^{x,m})|dt\, \big|\, \mathcal F_{t_1}\big] \leq C(t_2-t_1).$$
Indeed,
\begin{align*}
\mathbf{E}\bigg[\int_{t_1}^{t_2}|\mathsf{f}(t,X_{s,t}^{x,m})|dt\, \big|\, \mathcal F_{t_1}\bigg] & =\mathbf{E}\bigg[u_m (t_1,X_{s,t_1}^{x,m})\,|\, \mathcal F_{t_1}\bigg] \\
& \leq \|u_m(t_1)\|_{\infty} \\
& (\text{we are applying \cite[Cor.\,6.4]{KM}}) \\
& \leq C'\sup_{z \in \mathbb Z^d}\|\mathsf{f} \sqrt{\rho_{z}}\|_{L^2([t_1,t_2],L^2)},
\end{align*}
where, recall, $\rho(x)  
=(1+\kappa |x|^2)^{-\theta}$ with $\theta>\frac{d}{2}$ and $\kappa>0$ fixed sufficiently small, 
$\rho_z(x)=\rho(x-z)$.
In turn, since $\mathsf{f} \in \mathbf{F}_\nu$, 
\begin{align*}
\|\mathsf{f} \sqrt{\rho_{z}}\|^2_{L^2([t_1,t_2],L^2)} & \leq \frac{\nu}{4}\int_{t_1}^{t_2} \langle \frac{|\nabla \rho_z|^2}{\rho_z}\rangle dt  \\
& \leq \frac{\nu}{4} (t_2-t_1)\|\nabla \rho /\sqrt{\rho}\|_2^2 \\
& (\text{we are using $|\nabla \rho| \leq \theta\sqrt{\kappa}\rho$,  $\|\sqrt{\rho}\|_2<\infty$}) \\
&  \leq C''(t_2-t_1).
\end{align*}
which gives us 1).

\medskip

2) As a consequence of estimate 1), one has, e.g.\,for every integer $r \geq 1$,
$$
\sup_{m} \sup_{x \in \mathbb R^d}\mathbb E\left|\int_{t_1}^{t_2} |\mathsf{f}(t,X^{x,m}_{s,t})|dt\right|^r \leq C_r(t_2-t_1)^{r},
$$
see proof in \cite[Cor.\,3.5]{ZZ} (first, one represents $\mathbb E\left|\int_{t_1}^{t_2} |\mathsf{f}(t,X^{x,m}_{s,t})|dt\right|^r$ as the expectation of a repeated integral over $\Delta(t_1,t_2)$,
cf.\,transition from \eqref{int_1} to \eqref{int_2}, and then uses 1) $r$ times).

\medskip

3) It follows from 2) (upon selecting $\mathsf{f}=b_m$) that
\begin{align*}
\mathbf{E}|X^{x,m}_{s,t_2}-X^{x,m}_{s,t_1}|^r & \leq C\mathbb E\left|\int_{t_1}^{t_2} |b_m(t,X^{x,m}_{s,t})|\right|^r + C |W_{t_2}-W_{t_1}|^r \\
& \leq C(t_2-t_1)^{ r}.
\end{align*}

4) In particular,
$$
\sup_m\sup_{0 \leq s \leq t \leq T} \sup_{x \in \mathbb R^d}\mathbf E|X_{s,t}^{x,m}|^r <\infty.
$$

5) One has (the Sobolev norm is in the $x$ variable)
$$
\sup_m\sup_{0\leq s \leq t \leq T}\sup_{y \in \mathbb R^d} \mathbf{E}\int_{B_1(y)} |\nabla_x X^{x,m}_{s,t}|^r dx<\infty.
$$
Indeed,
\begin{align*}
\mathbf{E}\int_{B_1(y)} |\nabla_x X^{x,m}_{s,t}|^r dx
& \leq \left(\int_{B_1(y)} \bigl(\mathbf{E} |\nabla_x X^{x,m}_{s,t}|^r\bigr)^2 dx \right)^{\frac{1}{2}} |B_1(y)|^{\frac{1}{2}} \\
& = c_d \|\nabla X^m_{s,t}\|^r_{L^{2r}(B_1(y),L^r(\Omega))},
\end{align*}
so it remains to apply Proposition \ref{prop2}(\textit{i}).

\medskip

Let now $r>d$. Combining 4) and 5), and using the Sobolev embedding theorem, one obtains (the H\"{o}lder norm is in the $x$ variable)
$$
\sup_m\sup_{0 \leq s \leq t \leq T} \sup_{y \in \mathbb R^d} \mathbf{E}\|X_{s,t}^{x,m}\|_{C^{1-\frac{d}{r}}(B_1(y))}<\infty,
$$ 
and so one arrives at:

\medskip

6) For all $0 \leq s \leq t \leq T$, $x,y \in \mathbb R^d$ with $|x-y| \leq 1$, $$\mathbf{E}|X^{x,m}_{s,t_2}-X^{y,m}_{s,t_1}|^r \leq C|x-y|^{r-d}, \quad r>d.$$

\medskip

7) Repeating the proof from \cite{RZ} (which is a combination of 3) and 6), by means of the Markov property and the independence of $X_{s_1,s_2}^{x,m}$ and $X_{s_2,t}^{y,m}$), one arrives at
$$
\mathbf{E}|X_{s_1,t}^{x,m}-X_{s_2,t}^{x,m}|^r \leq C(s_2-s_1)^{r-d}, \quad 0 \leq s_1 \leq s_2 \leq t.
$$

\medskip

Estimates 3), 6), 7) combined yield, for $r>d$,
\begin{equation}
\label{ineq1}
\mathbf{E}|X_{s_1,t_1}^{x,m}-X_{s_2,t_2}^{y,m}|^r \leq C(|t_2-t_1|^{r} + |x-y|^{r-d} + |s_2-s_1|^{r-d})
\end{equation}
for all $(s_i,t_i) \in \Delta_2(T)$, $i=1,2$.

\medskip

Now comes the final stage in the approach of R\"{o}ckner-Zhao. Proposition \ref{prop2}(\textit{i})-(\textit{iii}) verifies conditions of \cite[Lemma 3.1]{RZ}, i.e.\,of the relative compactness criterion for  random fields on the Wiener-Sobolev space (see the discussion of  history of this type of results in \cite{RZ}). This, and a standard diagonal argument, allow to conclude that there is a subsequence of $\{X_{s,t}^{x,m}\}$ (without loss of generality, still denoted by $\{X_{s,t}^{x,m}\}$) and a countable subset $D$ of $\mathbb R^d$ such that
$$
X_{s,t}^{x,m} \rightarrow X_{s,t}^x \quad \text{ in } L^2(\Omega) \quad \text{ as } m \rightarrow \infty
$$
for all $(s,t) \in \mathbb Q^2 \times \Delta_2(T)$, $x \in D$.
Moreover, in view of 4), one has  $X_{s,t}^{x,m} \rightarrow X_{s,t}^x$ in $L^r(\Omega)$, $r \geq 1$. Now \eqref{ineq1} yields for $r>d$, upon applying Fatou's lemma, that
\begin{equation}
\label{ineq2}
\mathbf{E}|X_{s_1,t_1}^{x}-X_{s_2,t_2}^{y}|^r \leq C(|t_2-t_1|^{r} + |x-y|^{r-d} + |s_2-s_1|^{r-d})
\end{equation}
for all $(s_i,t_i) \in \mathbb Q^2 \times \Delta_2(T)$, $i=1,2$, $x,y \in D$. Kolmogorov-Chentsov theorem (after selecting $r>d$ even larger) allows to extend $X_{s,t}^x$ to a continuous random field, and yields, together with the equicontinuity estimate \eqref{ineq1},
$$
X_{s,t}^{x,m} \rightarrow X_{s,t}^{x} \quad \text{$\mathbf{P}$-a.s.}  \text{ as } m \rightarrow \infty
$$
for all $(s,t) \in \Delta_2(T)$, $x \in \mathbb R^d$.

By \cite[Cor.\,6.4]{KM},
\begin{equation}
\label{bd_}
\mathbf{E}\bigg[\int_{s}^{t}|\mathsf{f}(\tau,X_{s,\tau}^{x,m})|d\tau\bigg] \leq C_1\sup_{z \in \mathbb Z^d}\|\mathsf{f} \sqrt{\rho_{z}}\|_{L^2([s,t],L^2)}
\end{equation}
(cf.\,proof of 1) above), and so
\begin{equation}
\label{bd_2}
\mathbf{E}\bigg[\int_{s}^{t}|\mathsf{f}(\tau,X_{s,\tau}^{x})|d\tau\bigg] \leq C_1\sup_{z \in \mathbb Z^d}\|\mathsf{f} \sqrt{\rho_{z}}\|_{L^2([s,t],L^2)}
\end{equation}
where, recall, $\mathsf{f} \in \mathbf{F}_\nu$ is bounded and smooth, but the constant $C_r$ does not depend on smoothness of bundedness of $\mathsf{f}$. Using Fatou's lemma, one can extend \eqref{bd_} to all $\mathsf{f} \in \mathbf{F}_\nu$, i.e.\,not necessarily smooth. (We will be selecting e.g.\,$\mathsf{f}=b-b_m$, in which case $\nu=2\delta$.)

Now, to show that $X_{s,t}^x$ is a strong solution to \eqref{sde2}, it remains to show that 
$\int_s^t b_m(\tau,X_{s,\tau}^{x,m})d\tau \rightarrow \int_s^t b(\tau,X_{s,\tau}^{x})d\tau$ in $L^1(\Omega)$. Indeed,
\begin{align*}
\mathbf{E}\biggl|\int_s^t (b_m(\tau,X_{s,\tau}^{x,m})d\tau -\int_s^t b(\tau,X_{s,\tau}^{x})d\tau \biggr| & \leq \mathbf{E}\biggl|\int_s^t (b_m-b_k)(\tau,X_{s,\tau}^{x,m})d\tau\biggr| \\
& + \mathbf{E}\biggl|\int_s^t b_k(\tau,X_{s,\tau}^{x,m})d\tau - \int_s^t b_k(\tau,X_{s,\tau}^{x})d\tau\biggr| \\
& + \mathbf{E}\biggl|\int_s^t (b_k-b)(\tau,X_{s,\tau}^{x})d\tau\biggr|=:I_1+I_2+I_3.
\end{align*}
By \eqref{bd_}, $I_1 \leq \sup_{z \in \mathbb Z^d}\|(b_m-b_k) \sqrt{\rho_{z}}\|_{L^2([s,t],L^2)} \rightarrow 0$ as $m,k \rightarrow \infty$, where the $L^2$ norm tends to zero since by our assumption $b$ has compact support. Let us fix $k$ sufficiently large. By \eqref{bd_2}, $I_3 \rightarrow 0$ as $m \rightarrow \infty$. Finally, $I_2 \rightarrow 0$ as $m \rightarrow \infty$ (for $k$ fixed above) by the Dominated convergence theorem. This yields  that $X_{s,t}^x$ is a strong solution to \eqref{sde1}. This strong solution, clearly, satisfies \eqref{krylov}.

Finally, regarding uniqueness of $X_{s,t}^x$ in the class of strong solutions  satisfying Krylov estimate \eqref{krylov}. The proof in \cite{RZ} is based on Cherny's theorem \cite{C} (strong existence + weak uniqueness $\Rightarrow$ strong uniqueness) and the result from \cite{RZ_weak} on the uniqueness of weak solution to SDE \eqref{sde1} in the class of solutions satisfying a Krylov-type bound. In our setting, it suffices to use instead the weak uniqueness result from \cite{KM}, valid for form-bounded drifts. This yields the uniqueness result in Theorem \ref{thm1} within the class \eqref{krylov0}; regarding the uniqueness within the class \eqref{krylov}, one needs to apply the weak uniqueness result from \cite{Ki_Morrey}.

\hfill \qed


\bigskip

\begin{thebibliography}{99}

\bibitem[B]{B} R.\,Bass, Diffusions and Elliptic Operators, Springer, 1998.

\bibitem[BFGM]{BFGM} L.~Beck, F.~Flandoli, M.~Gubinelli and M.~Maurelli, 
\newblock Stochastic ODEs and stochastic linear PDEs with critical
drift: regularity, duality and uniqueness. 
\newblock{\em Electr. J. Probab.}, \textbf{24} (2019),  Paper No. 136, 72 pp (arXiv:1401.1530).

\bibitem[C]{C} A.\,S.\,Cherny, On the uniqueness in law and the pathwise uniqueness for stochastic
differential equations, {\em Theory of Probability and its Applications}, \textbf{46}(3) (2002), 406-419.

\bibitem[CFr]{CF}  F. Chiarenza and M. Frasca, A remark on a paper by C. Fefferman, {\em Proc. Amer. Math. Soc.}, \textbf{108} (1990),
407-409. 

\bibitem[F]{F}  C. Fefferman, The uncertainty principle, {\em Bull. Amer. Math. Soc.} \textbf{9} (1983), 129-206.

\bibitem[K]{Ki_Morrey} D.\,Kinzebulatov, Parabolic equations and SDEs with time-inhomogeneous Morrey drift, arXiv:2301.13805.


\bibitem[K2]{Ki_survey} D.\,Kinzebulatov, Form-boundedness and SDEs with singular drift, arXiv:2305.00146.

\bibitem[KM]{KM} D.Kinzebulatov and K.R.Madou, Stochastic equations with time-dependent singular drift, {\em J.\,Differential Equations}, \textbf{337} (2022), 255-293 (arXiv:2105.07312).


\bibitem[KS]{KiS_sharp} D.\,Kinzebulatov and Yu.\,A.\,Sem\"{e}nov, Sharp solvability for singular SDEs, {\em Electr.\,J.\,Probab.}, to appear (arXiv:2110.11232).

\bibitem[KS2]{KiS_brownian} D.\,Kinzebulatov and Yu.A.\,Sem\"{e}nov, Brownian motion with general drift, \newblock{\em Stoch. Proc. Appl.}, \textbf{130} (2020), 2737-2750 (arXiv:1710.06729).


\bibitem[KSS]{KiSS_transport} D.\,Kinzebulatov, Yu.\,A.\,Sem\"{e}nov and R.\,Song, Stochastic transport equation with singular drift, {\em Ann.\,Inst.\,Henri Poincar\'{e} (B) Probab. Stat.}, to appear (arXiv:2102.10610).



\bibitem[Kr1]{Kr1} N.\,V.\,Krylov,  \newblock  On diffusion processes with drift in $L_d$, \newblock{\em Probab. Theory Relat. Fields} 179, 165-199 (2021) (arXiv:2001.04950). 

\bibitem[Kr2]{Kr2} N.\,V.\,Krylov, \newblock On strong solutions of It\^{o}'s equations with $A \in W^{1,d}$ and $B \in L^d$, arXiv:2007.06040.



\bibitem[Kr3]{Kr3} N.V.\,Krylov, On strong solutions of It\^{o}'s equations with $D\sigma$ and $b$ in Morrey classes containing $L^d$, arXiv:2111.13795.

\bibitem[Kr4]{Kr_strong} N.V.\,Krylov, On strong solutions of time-inhomogeneous It\^{o}'s equations with Morrey diffusion gradient and drift. A supercritical case, arXiv:2211.03719v2.

\bibitem[Kr5]{Kr_Morrey_sdes} N.V.\,Krylov, Once again on weak solutions of time-inhomogeneous It\^{o}'s equations with VMO diffusion and Morrey drift, arXiv:2304.04634.




\bibitem[KrR]{KR} N.\,V.\,Krylov and M.\,R\"{o}ckner, 
\newblock Strong solutions of stochastic equations with singular time dependent drift. 
\newblock {\em Probab. Theory Related Fields}, 131 (2005), 154-196.


\bibitem[MP]{MP} T.\,Meyer-Brandis and F.\,Proske, Construction of strong solutions of SDE's via Malliavin calculus, {\em J.\,Funct.\,Anal.}, \textbf{258} (2010)(11), 3922-3953.

\bibitem[MNP]{MNP} A.\,Mohammed, T.\,Nilssen and F.\,Proske, Sobolev differentiable stochastic flows
for SDEs with singular coefficients: applications to the transport equation, {\em Ann.\,Probab.},
\textbf{43}(3) (2015), 1535-1576.

\bibitem[R]{R} F.\,Rezakhanlou, {Regular flows for diffusions with rough drifts}, arXiv:1405.5856.

\bibitem[RZ]{RZ} M.\,R\"{o}ckner and G.\,Zhao,  SDEs with critical time dependent drifts: strong solutions,
arXiv:2103.05803.

\bibitem[RZ2]{RZ_weak} M.\,R\"{o}ckner and G.\,Zhao, {SDEs with critical time dependent drifts: weak solutions}, {\em Bernoulli}, \textbf{29} (2023), 757-784 (arXiv:2012.04161).


\bibitem[S]{S} Yu.\,A.\,Sem\"{e}nov, \newblock Regularity theorems for parabolic equations, \newblock {\em J.\,Funct.\,Anal.}, \textbf{231} (2006), 375-417.

\bibitem[V]{V} A.\,Yu.\,Veretennikov, Strong solutions and explicit formulas for solutions of stochastic integral equations, {\em Matematicheski Sbornik} (in
Russian), \textbf{111}(3) (1980), 434-452, English translation in {\em Math. USSR-Sbornik}, \textbf{39}(3) (1981), 387-403.


\bibitem[ZZ]{ZZ} X.\,Zhang and G.\,Zhao, Singular Brownian diffusion processes, {\em Communications in Mathematics and Statistics}, November 2018, DOI: 10.1007/s40304-018-0164-7. 


\end{thebibliography}
\end{document}